\documentclass[11pt,reqno]{article}

\usepackage[usenames]{color}
\usepackage{graphicx}
\usepackage{amscd}
\usepackage[english]{babel}
\usepackage{color}
\usepackage{fullpage}
\usepackage{psfig}
\usepackage{graphics,amsmath,amssymb}
\usepackage{amsthm}
\usepackage{amsfonts}
\usepackage{latexsym}
\usepackage{epsf}
\usepackage{float}
\usepackage{MnSymbol}

\usepackage[colorlinks=true,
linkcolor=webgreen,
filecolor=webbrown,
citecolor=webgreen]{hyperref}

\definecolor{webgreen}{rgb}{0,.5,0}
\definecolor{webbrown}{rgb}{.6,0,0}

\newcommand{\pl}{\parallel}  
\newcommand{\md}{\leftrightline}  

\theoremstyle{plain}
\numberwithin{equation}{section}
\newtheorem{thm}{Theorem}[section]
\newtheorem{theorem}[thm]{Theorem}

 \newcommand{\seqnum}[1]{\href{http://oeis.org/#1}{\underline{#1}}}

\newcommand{\eqn}[1]{(\ref{#1})}

\newcommand{\beql}[1]{\begin{equation}\label{#1}}
\newcommand{\eeq}{\end{equation}}

\begin{document}


\setcounter{page}{1}


\begin{center}

{\large\bf ``Choix de Bruxelles'': A New Operation on Positive Integers } \\
\vspace*{+.2in}

\'{E}ric Angelini \\
32 dr\`{e}ve de Linkebeek 1640 Rhode-Saint-Genese, BELGIUM \\
Email:  \href{mailto:eangelini@everlastingprod.be}{\tt eangelini@everlastingprod.be}

\vspace*{+.1in}

Lars Blomberg \\
\"{A}renprisv\"{a}gen 111,
SE-58564 Linghem, SWEDEN\\
Email:  \href{mailto:larsl.blomberg@comhem.se}{\tt larsl.blomberg@comhem.se}

\vspace*{+.1in}

Charlie Neder \\
1122 Clement St.,
 Watertown, WI 53094, USA\\
 Email:  \href{mailto:Charlie.Neder@mbu.edu}{\tt Charlie.Neder@mbu.edu}

\vspace*{+.1in}

R\'{e}my Sigrist \\
3 rue de la Somme,
 67000 Strasbourg, FRANCE \\
 Email:  \href{mailto:remyetc9@gmail.com}{\tt remyetc9@gmail.com}
  
\vspace*{+.1in}
\
N. J. A. Sloane\footnote{To whom correspondence should be addressed.} \\
The OEIS Foundation Inc., 
11 South Adelaide Ave.
Highland Park, NJ 08904, USA \\
Email:  \href{mailto:njasloane@gmail.com}{\tt njasloane@gmail.com}
\vspace*{+.1in}

\begin{abstract}
The ``Choix de Bruxelles'' operation replaces a positive integer $n$ by any of
the numbers that can be obtained by halving or doubling a substring of the decimal representation of $n$.
For example, $16$ can become any of $16$, $26$, $13$,  $112$, $8$, or $32$.
We investigate the properties of this interesting operation and its iterates.
\end{abstract}
\end{center}


\section{Introduction}\label{Sec1}
Let $n$ be a positive integer with
decimal expansion $d_1 d_2 d_3 \ldots d_k$.
The ``Choix de Bruxelles'' operation replaces  $n$ by any of
the numbers that can be formed by taking a number $s$
represented by a substring $d_p d_{p+1} \ldots d_q$, 
with $1 \le p \le q \le k$,  where the initial digit  $d_p$ 
is not zero,  and replacing the substring \emph{in situ}
by the decimal expansion of $2s$ or, if $s$ is even,
by the decimal expansion of $s/2$.
One may also leave $n$ unchanged (corresponding to choosing the empty substring). 

Since the definition may be confusing at first glance, we give a detailed example.
Suppose $n = 20218$. By choosing substrings of length $1$, we can obtain any of
$$
10218, 40218, 20118, 20418, 20228, 20214, 
\text{and~} 202116 (!).
$$
(For the last of these, we replaced $8$ by $16$ \emph{in situ}.) 
Using substrings of length $2$ we can obtain
$$
10218 ~ \& ~ 40218 \text{~again}, 20428, 2029 (!), 20236,
$$
and using substrings of lengths $0, 3, 4, 5$ we obtain
$$
20218; 10118, 40418, 20109, 20436; 40428; 10109 \text{~and~} 40436.
$$
It is the possibility of increasing or decreasing the number of
digits by one that makes this operation interesting. Changing $16$ to $112$,
for example, increases the number by a factor of $7$.
In fact, as we see in Theorem~\ref{th3}, $n$ may
change by a factor ranging from $\frac{1}{10}$ to $10$.
Note that the operation is symmetric: if $n$ can be changed to $m$, then $m$
can be changed to $n$.

The name ``Choix de Bruxelles'' arose from a combination of several ideas:
this a kind of mathematical game, like ``sprouts'' \cite{WW},
sprouts are ``choux de Bruxelles'',  the operation was proposed by the first author,
who lives in Brussels, and involves making choices.

Our goal is to investigate the properties of this operation, and to study what happens
when it is iterated.

Section~\ref{Sec2} studies the range of numbers that can be reached by applying
the operation once.  Theorems~\ref{th1} and \ref{th2} 
 give the largest and smallest numbers that can be obtained from $n$, and Theorem~\ref{th3}  
 describes their range.
 
 Section~\ref{Sec3} shows that by repeatedly applying the operation, any
 number not ending with $0$ or $5$ can be transformed into
 any other such number, and any number ending in $0$ or $5$ can be be transformed into any other
 number of that form, but the two classes always remain separate. The final section studies
 how many steps are need to reach $n$ from $1$, assuming
 that $n$ does not end in $0$ or $5$.

One could also consider this operator in other bases. However, the binary version
(see sequence  \seqnum{A323465}\footnote{Six-digit numbers prefixed by A refer to entries in the On-Line Encyclopedia of Integer Sequences.}
 in \cite{OEIS})
is not very interesting---the operation preserves the binary weight---so we have not pursued this.

Notation.  For clarity we will sometimes use the symbol $\pl$ to
indicate that the digits of two numbers are
to be concatenated.  For example, if $x=7$ and $y=8$, $(2x) \pl y$ 
indicates the number $148$.

\begin{table}[htb]
\caption{Numbers arising when ``Choix de Bruxelles'' is applied to the numbers 1 to 16.}\label{Tab1}
$$
\begin{array}{|r|c||r|c|}
\hline
n & \text{goes~to} & n & \text{goes~to} \\
\hline
1 & 1, 2 &     9 & 9, 18 \\
2 & 1, 2, 4 & 10 & 5, 10, 20 \\
3 & 3, 6    & 11 & 11, 12, 21, 22 \\
4 & 2, 4, 8 & 12 & 6, 11, 12, 14, 22, 24 \\
5 & 5, 10   & 13 & 13, 16, 23, 26 \\
6 & 3, 6, 12& 14 & 7, 12, 14, 18, 24, 28 \\
7 & 7, 14   & 15 & 15, 25, 30, 110 \\
8 & 4, 8, 16& 16 & 8, 13, 16, 26, 32, 112 \\
\hline
\end{array}
$$
\end{table}

\section{Numbers that can be reached in one step.}\label{Sec2}

Table \ref{Tab1} shows the numbers that can be reached from $n$ in one step, for $1 \le n \le 16$.
The rows of this table form  
sequence \seqnum{A323460} in the OEIS \cite{OEIS}. 

The smallest ($M_0(n)$, \seqnum{A323462}) and largest ($M_1(n)$, \seqnum{A323288}) numbers in each row are as follows:
\beql{EqMM}
   \begin{array}{rrrrrrrrrrrrrrrrrrrr}
   n: & 1 & 2 & 3 & 4 & 5 & 6 & 7 & 8 & 9 & 10 & 11 & 12 & 13 & 14 & 15 & 16 &   \cdots \\
M_0(n): & 1 & 1 & 3 & 2 & 5 & 3 & 7 & 4 & 9 & 5 & 11 & 6 & 13 & 7 & 15 & 8 &  \cdots \\
M_1(n): & 2 & 4 & 6 & 8 & 10 & 12 & 14 & 16 & 18 & 20 & 22 & 24 & 26 & 28 & 110 & 112 &  \cdots
\end{array}
\eeq

\begin{theorem}\label{th1}
The largest number $M_1(n)$ that can be obtained from
\beql{Eqddd}
 n ~=~ d_1 d_2 d_3 \ldots d_k
 \eeq
 by the  Choix de Bruxelles operation is either
$2n$,  if all $d_i <5$, and otherwise is obtained by doubling the 
substring $d_p \ldots d_k$ where $d_p$ is the right-most digit $\ge 5$.
\end{theorem}

\begin{proof}
If all $d_i < 5$, the number of digits
will not change during the operation, and so the best we can do is to double every digit, getting
$M_1(n)=2n$.

If there are digits $d_i \ge 5$, then doubling any substring
beginning with such a $d_i$ will increase the number of digits by one.
The doubled substring will 
begin with $1$ instead of $d_i$, and so should be as far to the right as possible.
(E.g., if we double a substring starting at $d_3 = 7$, say, we 
get a number $d_1 d_2 1 \ldots$,
whereas if we double a substring starting at $d_5=6$ we get $d_1 d_2 d_3 d_4 1 \ldots$,
which is a larger number since $d_3 \ge 5$.)

Finally, if $d_p$ is the right-most digit $\ge 5$, let $t_q$  (where $p \le q \le k$)
denote the result of doubling 
just the substring  from $d_p$ through $d_q$. Then 
$$
t_q ~=~ d_1 \ldots d_{p-1}  \pl 2(d_p \ldots d_q) \pl d_{q+1} \ldots d_k,
$$
and this number is maximized by taking $q=k$.
\end{proof}

\begin{theorem}\label{th2}
The smallest number $M_0(n)$ that can be obtained from \eqn{Eqddd}
 by the  Choix de Bruxelles operation is as follows:
 If there is a substring $d_r ... d_s$ of \eqn{Eqddd} which starts with $d_r = 1$ and ends with an even digit $d_s$,
take the string of that form which starts with the left-most $1$ and ends with the right-most even digit, and halve it. Otherwise, if there is any even digit, take the substring from $d_1$ to the right-most even digit and halve it. 
In the remaining case, all $d_i$ are odd, and $M_0(n) = n$.
\end{theorem}

 The proof uses similar arguments to the proof of Theorem~\ref{th1}, and is omitted.
  
  We can use Theorems~\ref{th1} and \ref{th2} to get precise bounds on the range
  of numbers produced by the operation.
  
  \begin{theorem}\label{th3}
  The numbers $m$ obtained by applying the Choix de Bruxelles operation to $n$ lie in the range
  \beql{Eq10}
  \frac{n}{10}~<~ m ~<~ 10n,
  \eeq
  and there are values of $n$ for which $m$ is
  arbitrarily close to either bound.
  \end{theorem}
  
  \begin{proof}
  For the upper bound, by Theorem~\ref{th1}  we can assume there is a digit of $n$ that is 
  $\ge 5$. Let $d$ be the right-most such digit, and write $n = A \pl d \pl B = A\,10^i+d\,10^{i-1}+B$.
  Let $2d=10+x$, $x \in \{0,2,4,6,8\}$. Then
  $$
  M_1(n) = A\,10^{i+1} + 10^i + x\,10^{i-1}+2B
  $$
  and $M_1(n)<10n$ follows immediately.
  Numbers of the form  $n=99 \ldots 9$ with $d=9$, $i=1$, $B=0$ have $M_1(n)=99\ldots918 = 10n-72$, 
  so $M_1(n)/n = 10-72/n$,  which comes arbitrarily close  to the bound.
  
  For the lower bound, from Theorem~\ref{th1} the only nontrivial case is when there is a substring
  $B = 1 \ldots e$, $e$ even, and then $n = A\,10^{i+j}+B\,10^j+C$,
  $M_0(n)=A\,10^{i+j-1}+\frac{B}{2} 10^j+C$, which implies
  $n<10\, M_0(n)$. Numbers of the form $n=10^t+10$, $t \ge 2$, with
  $M_0(n) = 10^{t-1}+5$, $M_0(n)/n = \frac{1}{10}(1+4/10^{t-1} + \cdots)$ come arbitrarily close to this bound.
  \end{proof}

\section{Which numbers can be reached from 1?}\label{Sec3}

If we start with $1$ and repeatedly apply the Choix de Bruxelles operation, which numbers can we reach?
From\footnote{We write $\md$ rather than $\rightarrow$  in describing these transformations, since the operation is symmetrical.}
\beql{EqGet3}
1 \md 2 \md 4 \md 8 \md 16 \md 112 \md 56 \md 28 \md 14 \md 12 \md 6 \md 3
\eeq
we can reach $3$ in $11$ steps\footnote{Found by Lorenzo Angelini.},
as well as $2, 4, 8$ and $6$, and $ \cdots \md 28 \md 14 \md 7$ and
$ \cdots \md 28 \md  14 \md 18 \md 9$ give us all the numbers less than $10$ except $5$.
Further experimenting suggests that $5$ and $10$ may be impossible to reach. This is true,
as we now show.

Let us define an undirected graph $G$ with vertices labeled by the positive integers,
where vertices $n$ and $m$ are joined by an edge if Choix de Bruxelles takes $n$ to $m$.

\begin{theorem}
The graph $G$ has two connected components, one containing all numbers whose decimal expansion does not end in $0$ or $5$,  the other   containing all numbers which do end in $0$ or $5$.
\end{theorem}
\begin{proof}
(i) All numbers not ending in $0$ or $5$ are connected to $1$. If not,
let $n$ be the smallest number, not ending in $0$ or $5$, that cannot be reached from 1.
Any number $m$ that is reachable from $n$ must be larger
than $n$, or else by symmetry $m$ would be a smaller
counter-example.
This means that all digits of $n$ are odd, and cannot be $3$, $7$, or $9$ (since by 
\eqn{EqGet3} etc. they could be reduced to $1$).
If $n$ has an internal digit $5$, we can use $51 \md 102 \md 52 \md 26$ to get a smaller number.
So any $5$s must appear in a string at the end of $n$, which contradicts our assumptions.
The only remaining possibility is that $n = 11\ldots 11$. But even these numbers can be 
reduced by using  $11 \md 12 \md 6$.

(ii) Any number ending in $0$ or $5$ is connected to $5$. This follows
from (i), using $n0 \md 10 \md 5$ and
$n5 \md n\pl10 \md 10 \md 5$.

(iii) The Choix de Bruxelles operations never change a final $0$ or $5$ into any digit other
than $0$ or $5$.
\end{proof}

\section{How many steps to reach \texorpdfstring{$n$}{n}?}\label{Sec4}

Let $\tau(n)$ be the number of steps to reach $n$ from $1$  using
the Choix de Bruxelles operation, or $-1$ if $n$ cannot be reached from $1$. 
We found the values of $\tau(n)$ for $n \le 10000$ by
computer (\seqnum{A323454}).
The initial values are shown in Table~\ref{Tab2}. 

\begin{table}[H]
\caption{ $\tau(n)=$ number of steps to reach $n$ from $1$, or $-1$ if $n$ cannot be reached. }
\label{Tab2}
$$
   \begin{array}{|r|rrrrrrrrrrrr|}
  \hline
n & 1 & 2 & 3 & 4 & 5 & 6 & 7 & 8 & 9 & 10 & 11 & 12 \\
\tau(n) & 0 & 1 & 11 & 2 & -1 & 10 & 9 & 3 & 9 & -1 & 10 & 9 \\
\hline 
n & 13 & 14 & 15 & 16 & 17 & 18 & 19 & 20 & 21 & 22 & 23 & 24 \\
\tau(n) & 5 & 8 & -1 & 4 & 7 & 8 & 8 & -1 & 10 & 9 & 6 & 8 \\
\hline 
n & 25 & 26 & 27 & 28 & 29 & 30 & 31 & 32 & 33 & 34 & 35 & 36 \\
\tau(n) & -1 & 5 & 8 & 7 & 9 & -1 & 6 & 5 & 10 & 6 & -1 & 9 \\
\hline
   \end{array}
  $$
\end{table}

We will derive  upper and lower bounds on $\tau(n)$ for large $n$.
The record high values of $\tau(n)$ that are presently known,
and the values of $n$ where they occur, are as follows (see \seqnum{A323463}):
\beql{EqHWtau}
   \begin{array}{rrrrrrrrr}
   n: & 1 & 2 & 3 & 99 & 369 & 999 & 1999 & 9879  \\
   \tau(n): & 0 & 1 & 11 & 12 & 13 & 14 & 15 & 16
   \end{array}
\eeq
In particular, the entry $\tau(n) = 12$  for  $n=99$  means that every number less than $100$ and not ending in $0$ or $5$ can
be reached from $1$ in at most $12$ steps.

We also need analogous data for the other class of numbers. The
 numbers of steps to reach $5m$ from $5$ (\seqnum{A323484})
are:

\beql{EqSteps5}
   \begin{array}{rrrrrrrrrrrrrr}
   5m: & 5 & 10 & 15 & 20 & 25 & 30 & 35 & 40 & 45 & 50 & 55 & 60 & \cdots \\
  \text{steps}:  & 0 & 1 & 11 & 2 & 11 & 12 & 8 & 3 & 10 & 12 & 9 & 11 & \cdots \\
      \end{array}
     \eeq
and the record high values  that are presently known,
and the values of $5m$ where they occur, are as follows (see \seqnum{A323464}):
\beql{EqHWSteps5}
 \begin{array}{rrrrrrrrrr}
     5m: & 5 & 10 & 15 & 30 & 100 & 200 & 400 & 9875 & 19995 \\
\text{steps}: & 0 & 1 & 11 & 12 & 13 & 14 & 15 & 16 & 17 \\
\end{array}
\eeq
In particular, the entry $13$ for $n=100$ means that every multiple of $5$ less than $100$ can
be reached from $5$ in at most $12$ steps, and so \emph{all} numbers less than $100$
can be reached from either $1$ or $5$ in at most $12$ steps.

To get an upper bound on $\tau(n)$  for larger $n$, consider a $k$-digit number $n$ given by~\eqn{Eqddd}.
We can repeatedly replace the two leading digits by a single digit (by $1$ or $5$, in fact)
at the cost of at most $12$ steps. So we can reach $1$ or $5$ in at most $12(k-1)$
steps. Since $k = \lfloor 1 + \log_{10} n \rfloor$, this takes at most $12 \log_{10} n$ steps.
In particular, $\tau(n) \le 12 \log_{10} n$.

To get a lower bound on $\tau(n)$, we note that Theorem~\ref{th3}  implies 
that at least $\log_{10} n$ steps will be needed to reach $n$ from $1$. 
We can make this more precise  by finding  $R(n)$,
the largest number that can be reached from $1$ in $n$ steps.  
The values of $R(n)$ can \emph{almost}  be found by a greedy algorithm.
Set $r(0)=1$, and, for $n>0$,
let $r(n)$ be the largest number that can be obtained  by applying Choix de Bruxelles to $r(n-1)$
(of course $R(n) \ge r(n)$).
Theorem~\ref{th1} tells us how to calculate $r(n)$, and the first $20$ values are shown in Table~\ref{Tabrrr}.
We find that  $r(n+4) = 8112 \pl r(n)$ for $n \ge 10$.

\begin{table}[htb]
\caption{ Values of $r(n)$. }
\label{Tabrrr}
$$
 \begin{array}{|c|r||c|r|}
 \hline
n & r(n) & n &  r(n) \\
\hline
0 & 1 & 10 &  88224 \\
1 & 2 & 11 &  816448 \\
2 & 4 & 12 &  8164416 \\
3 & 8 & 13 &  81644112 \\
4 & 16 & 14 &  811288224 \\
5 & 112 & 15 &  8112816448 \\
6 & 224 & 16 &  81128164416 \\
7 & 448 & 17 &  811281644112 \\
8 & 4416 & 18 &  8112811288224 \\
9 & 44112 & 19 &  81128112816448 \\
\hline
\end{array}
$$
\end{table}

\begin{theorem}\label{ThRRR}
$R(n)$, the largest number that can be reached from $1$ in $n$ steps,
is equal to $r(n)$ for $n \ne 7$, and $R(7)=512$.
\end{theorem}
\begin{proof}
By computer calculation, the result is true for $n \le 16$.
Suppose $n \ge 17$.
The candidates for $R(n)$ are all the numbers that can be
obtained by applying Choix de Bruxelles to the $(n-1)$st generation numbers---numbers that
can be reached from $1$ in $n-1$ steps. 
In view of Theorem~\ref{th3},
we can discard any  $(n-1)$st generation numbers that are less than $\frac{1}{10} r(n)$.
The sets of remaining candidates form a repeating pattern of period four, 
which, starting in generation $14$, are as follows
(here  $i=1, 2, 3, \cdots$ and $P$ denote $i$ copies of the string $8112$):
\begin{align}\label{Eqleft}
\text{Generation} ~&~ \text{Remaining candidates} \nonumber \\
10+4i ~&~  P81646, P81652, P81662, P81664, P84112, P88112, P88212, P88222, P88224 \nonumber \\
11+4i ~&~  P816442, P816444, P816448 \nonumber \\
12+4i ~&~  P8164416 \nonumber \\
13+4i  ~&~ P81128826, P81128832, P81644112 \nonumber  
\end{align}
At each new generation, the largest number is obtained by expanding the final number
in each row (using Theorem~\ref{th3}), and the resulting numbers are the $r(n)$, as claimed.
\end{proof}

In summary, we have:
 
 \begin{theorem}\label{thR}
 For $n \ge 14$,
 \beql{EqRBDS}
 8.112~10^{n-6} ~<~ R(n) ~\le~ 8.113 ~ 10^{n-6}.
 \eeq
 \end{theorem}

The upper bound in Theorem~\ref{thR} states that, starting at $1$, we cannot reach a number 
 greater than $c\,10^{n-6}$ in $n$ steps, where $c=8.113$. 
 By solving $c\,10^{n-6} = m$, we see that to reach $m$ from $1$ we require at least
 $n = 6 + \log_{10} (m/c) = \log_{10} m +5.09\ldots$ steps.
 
 By combining this with our earlier result, we have our final theorem.
 
 \begin{theorem}\label{thTAU}
 The number of steps needed to reach $n$ from $1$ by the Choix de Bruxelles operation, assuming $n$ does not end in $0$ or $5$,
 is bounded by
 \beql{EqTAU}
 \log_{10} n + 5 ~<~ \tau(n) ~\le~ 12 \log_{10} n.
 \eeq
 \end{theorem}



\bigskip
\hrule
\bigskip

Concerned with sequences
\seqnum{A323288},
\seqnum{A323453},
\seqnum{A323454},
\seqnum{A323460},
\seqnum{A323462},
\seqnum{A323463},
\seqnum{A323464},
\seqnum{A323465},
\seqnum{A323484}.

\bigskip
\hrule
\bigskip

\noindent 2010 Mathematics Subject Classification 11A63, 11B37

\end{document}